\documentclass[12pt]{article}
\usepackage{amsmath,amsfonts,amssymb,amsthm}
\usepackage{tcolorbox}
\definecolor{pantone312}{HTML}{009DD1}
\usepackage{mathtools}
\usepackage{bbm}
\usepackage{graphicx}
\usepackage{tabularx}
\usepackage{color}
\definecolor{darkgreen}{rgb}{0,0.55,0}
\newcommand{\grad}{\nabla}

\newcommand{\laplace}{\Delta}

\renewcommand{\div}{\grad\cdot}

\newcommand{\E}{\mathbbm{E}}
\newcommand{\R}{\mathbbm{R}}

\usepackage{sectsty}

\sectionfont{\large}
\subsectionfont{\normalsize}
\subsubsectionfont{\normalsize}
\paragraphfont{\normalsize}

\newcommand{\on}{\omega^{\nu}}
\newcommand{\un}{u^{\nu}}

\newcommand{\dd}{{\rm d}}
\newcommand{\dx}{\,\dd x}
\newcommand{\dy}{\,\dd y}
\newcommand{\dt}{\, \dd t}

\def\XXint#1#2#3{{\setbox0=\hbox{$#1{#2#3}{\int}$ }
\vcenter{\hbox{$#2#3$ }}\kern-.59\wd0}}

\newtheorem{theorem}{Theorem}
\newtheorem{lemma}{Lemma}

\pagecolor{white}

\begin{document}
\phantom{ }
\vspace{4em}

\begin{flushleft}
{\large \bf A note on the vanishing viscosity limit in the Yudovich class}\\[2em]
{\normalsize \bf Christian Seis}\\[1em]

{\bf Abstract:} We consider the inviscid limit for the two-dimensional Navier--Stokes equations in the class of integrable and bounded vorticity fields. It is expected that  the difference  between the Navier--Stokes and Euler velocity fields vanishes in   $L^2$ with an order proportional to the square root of the viscosity constant $\nu$. Here, we provide an order $\left(\nu/|\log\nu|\right)^{\frac12\exp(-Ct)}$ bound, which slightly improves upon earlier results by Chemin~\cite{Chemin96}.\\[1em]
\end{flushleft}

\vspace{2em}

\section{Introduction}
The convergence of solutions of the Navier--Stokes equations towards solutions of the Euler equations in the limit of vanishing viscosity is a topic of ongoing research for many years. Most of the progress has been made in the two-dimensional full space, in which both vortex stretching and boundary effects are absent. Configurations in which the vorticity field is non-smooth are of particular interest, as these include the important examples of vortex patches. In the present paper, we study the inviscid limit for integrable and bounded vorticity fields.

To be more specific, we are interested in the rate of $L^2$ convergence of the Navier--Stokes velocity fields towards the Euler velocity fields. To the best of our knowledge, the best estimate available in the literature is due to Chemin  \cite{Chemin96}, who provides an $O(\nu^{\frac12\exp(-Ct)})$ bound on the velocity difference. In the present paper, we slightly improve this result  by a logarithm,
\begin{equation}\label{21}
\|\un(t)-u(t)\|_{L^2} =O\left(\left(\frac{\nu}{|\log{\nu}|}\right)^{\frac12\exp(-Ct)}\right),
\end{equation}
as $\nu\ll1$.  Moreover, for small times $t\ll 1/|\log \nu|$, we obtain
\begin{equation}\label{7}
\|\un(t)-u(t)\|_{L^2} =O( \sqrt{\nu}).
\end{equation}

Better convergence results are available in the literature under additional regularity assumptions. For instance, in \cite{ConstantinWu95,ConstantinWu96}, Constantin and Wu obtained $O(\sqrt{\nu})$ convergence \emph{globally} in time under additional gradient bounds on the velocity field. Such bounds  are true, for instance,  for vortex patches with smooth boundaries. For these particular solutions, however, better rates can be obtained. Indeed, Abidi and Danchin \cite{AbidiDanchin04} established an $O(\nu^{\frac34})$ estimate for vortex patches with smooth boundaries and showed the optimality of this convergence order. These results were further generalized and extended (to higher order Sobolev velocity fields) by Masmoudi \cite{Masmoudi07}. We also mention the $L^{\infty}$ bounds by Cozzi in the case of bounded but not necessarily decaying velocity fields \cite{Cozzi09,Cozzi14}.

We will work on the Navier--Stokes and Euler equations in vorticity formulations. The (scalar) vorticity fields are computed as the rotations of the velocity vectors and are denoted by $\on$ in the case of the Navier--Stokes  and $\omega$ in the case of the Euler equations. The Navier--Stokes equation in vorticity formulation is the advection-diffusion equation
\begin{equation}\label{1}
\partial_t \omega^{\nu} + u^{\nu}\cdot \grad \omega^{\nu} = \nu \laplace \on,
\end{equation}
which reduces to the Euler vorticity equation, a simple advection equation,
\begin{equation}\label{2}
\partial_t \omega + u\cdot\grad \omega= 0,
\end{equation}
in the inviscid limit when $\nu\to 0$. The velocity vector fields can be  reconstructed from the vorticities via the Biot--Savart law
\[
\un = K\ast \on,\quad u = K\ast \omega, \quad \mbox{where } K(x) = \frac{1}{2\pi}\frac{x^{\perp}}{|x|^2},
\]
and   $x^{\perp} $ is the counterclockwise rotation by 90 degrees of a point  $x $ in the plane. We note for completeness that the velocity fields are divergence-free by construction, thus $\div \un=\div u=0$,
which is the incompressibility assumption on the fluid. 

Both evolution equations have to be equipped with an initial condition. In the following, we assume that the initial vorticities are identical, integrable and bounded, that is,
\begin{equation}\label{6}
\on(0)  = \omega(0) = \omega_0\in L^1(\R^2)\cap L^{\infty}(\R^2).
\end{equation}
These assumptions are retained by the evolution, in the sense that for any time $t$,
\begin{gather}
\|\on(t)\|_{  L^1} \le \|\omega_0\|_{L^1} ,\quad\|\on(t)\|_{ L^{\infty}} \le \|\omega_0\|_{L^{\infty}},\label{104}\\
\|\omega(t)\|_{ L^1} = \|\omega_0\|_{L^1} ,\quad\|\omega (t)\|_{ L^{\infty}} = \|\omega_0\|_{L^{\infty}}.\label{105}
\end{gather} 

We recall that in the class of integrable and bounded solutions,  both the Navier--Stokes and   Euler equations admit a unique global solution in the two-dimensional setting. Indeed, the well-posedness of the Navier--Stokes equations holds true under more general assumptions, see, e.g., the work of Ben-Artzi \cite{Ben-Artzi94} for a proof in the $L^1$ setting. Roughly speaking, the results for \eqref{1} are a consequence of the parabolicity of the equation. In the case of the Euler equations, well-posedness for initial data in the class \eqref{6} was first obtained by Yudovich  \cite{Yudovich63} and is essentially open for unbounded vorticities.
Yudovich's result was later recovered by Loeper  \cite{Loeper06a}, who deduced uniqueness for the Euler equations from stability estimates in terms of the $2$-Wasserstein distance. In either work, the central key for proving uniqueness for \eqref{2} is a $\log$-Lipschitz estimate for the velocity field, which is valid in the Yudovich class $L^1\cap L^{\infty}$, namely
  \begin{equation}\label{4}
|u(x) - u(y)|\le C |x-y|\left(1+ \log\left(1+\frac1{|x-y|}\right)\right),
\end{equation}
see, e.g., \cite[Lemma 8.1]{MajdaBertozzi02}. In this estimate, the constant $C$ depends on $\|\omega\|_{L^1}$ and $\|\omega\|_{L^{\infty}}$.

In our derivation of the bounds on the convergence order \eqref{21} and \eqref{7}, we build up on  Loeper's approach. More precisely, we derive an estimate on certain  $2$-Wasserstein distances, which provides, on the one hand,  a bound on the $1$-Wasserstein distance between the viscous and the inviscid vorticity fields, and, on the other hand, the desired bound on the $L^2$ norm of the distance of the corresponding velocity fields. We first state and discuss the latter. A definition of Wasserstein distances will be given in the subsequent section.

\begin{theorem}\label{T2}
For $\nu\ll1 $ and $t\ll \frac1{|\log \nu|}$, it holds that
\[
\|\un(t)-u(t)\|_{L^2}\lesssim  \sqrt{\nu t} .
\]
Moreover, for  any  fixed $t>0$, it holds
\[
\|\un(t)-u(t)\|_{L^2} = O\left( \left(\frac{\nu}{|\log\nu|}\right)^{\frac12\exp(-Ct)}\right)
\]
as $\nu\to0$, where $C >0$ is a constant dependent only on $\|\omega_0\|_{L^1}$ and $\|\omega_0\|_{L^{\infty}}$. 
\end{theorem}

Here and in the following, we write $A\lesssim B$ if there exists a constant $\Lambda$ independent of $\nu$ and $t$ such that $A\le \Lambda B$. Moreover, we write $A\sim B$ if both $A\lesssim B$ and $B\lesssim A$. Finally, $A\ll B$ means that $A\le \Lambda B$ for some sufficiently large $\Lambda$.

As mentioned earlier, up to the logarithmic improvement, our second estimate  has been established earlier by Chemin \cite{Chemin96}. Our contribution here is essentially a new proof that is based on stability estimates. With regard to the scaling in $\nu$, the first estimate  is supposedly optimal, even globally in time.

The results can be restated as bounds on the  homogeneous $H^{-1}$ norm of the vorticity fields, for instance, 
\begin{equation}\label{8}
\|\on(t) - \omega(t)\|_{\dot H^{-1}} = O\left( \left(\frac{\nu}{|\log\nu|}\right)^{\frac12\exp(-Ct)}\right),
\end{equation}
and complement thus a recent work by Constantin, Drivas and Elgindi \cite{ConstantinDrivasElgindi19}, in which the convergence in  $L^p$  is proved for vorticity fields in the Yudovich class,
\[
\lim_{\nu\to 0} \sup_t \|\omega(t) - \omega(t)\|_{L^p}=0.
\]
Note that there can be no rates of strong convergence without imposing additional regularity assumptions on the data, which can be easily seen on the linear level, cf.~Example 2 in \cite{Seis18}.

The results of our second theorem are similar in spirit to \eqref{8}, in the sense that they provide  estimates on a negative Sobolev norm.

\begin{theorem}\label{T1}
For $\nu\ll1 $ and $t\ll \frac1{|\log \nu|}$, it holds that
\[
W_1(\on(t),\omega(t))\lesssim  \sqrt{\nu t} .
\]
Moreover, for  any  fixed $t>0$, it holds
\[
W_1(\on(t),\omega(t)) = O\left( \left(\frac{\nu}{|\log\nu|}\right)^{\frac12\exp(-Ct)}\right)
\]
as $\nu\to0$, where $C >0$ is a constant dependent only on $\|\omega_0\|_{L^1}$ and $\|\omega_0\|_{L^{\infty}}$. 
\end{theorem}
Notice that, by the Kantorovich--Rubinstein theorem, the $1$-Wasserstein distance $W_1$ is  dual to the homogeneous Lipschitz norm $\|\cdot \|_{\dot W^{1,\infty}}$, cf.~\eqref{9} below. Moreover, as Wasserstein distances metrize weak convergence, cf.\ \cite[Theorem 7.12]{Villani03}, estimates on the Wasserstein distance translate into  estimates on the convergence order. That is, our second theorem shows that the vorticity fields of the viscous fluid, $\on$, converge towards the vorticity field of the inviscid fluid, $\omega$, weakly with order $O\left( \left({\nu}/{|\log\nu|}\right)^{\frac12\exp(-Ct)}\right)$. Again, the author believes that $O(\sqrt{\nu})$ convergence is optimal.
 In this regard, the situation is very similar to the inviscid limit problem for  linear advection--diffusion equations in the DiPerna--Lions setting considered earlier by the author by using new stability estimates for the continuity equation, see \cite{Seis17,Seis18}. (In a certain sense, the estimates in \cite{Seis17} are the linear analogues of Loeper's estimates for the 2D Euler equations \cite{Loeper06a}.)

The remainder of the paper is organized as follows. In the following section, we recall the definition of the Wasserstein distances and collect a number of properties, that will be useful in our proofs. The proofs of Theorems \ref{T2} and \ref{T1} will be provided in the last section.

\section{Some tools from the theory of optimal transportation}
In this section, we collect definitions and properties of Wasserstein distances that will be used in the sequel. For a general comprehensive introduction into the topic of optimal transportation, we refer to Villani's popular monograph \cite{Villani03}.

 Given two nonnegative integrable functions $f$ and $g$ of the same total mass,
\begin{equation}\label{10}
\int_{\R^2} f\dx = \int_{\R^2} g\dx,
\end{equation}
we define the set of transport plans $\Pi(f,g)$ as the set of joint measures $\pi$ on the product space $\R^2\times\R^2$ having $f$ and $g$ as marginals, that is,
\[
\int_{\R^2\times\R^2}\left(\varphi(x) + \psi(y)\right)\,  \dd \pi(x,y) = \int_{\R^2} \varphi(x) f (x)\dx + \int_{\R^2}\psi(y)g(y)\dy,
\]
for any continuous functions $\varphi$ and $\psi$. The $p$-Wasserstein distance $W_p(f,g)$ between $f$ and $g$ is then defined by the formula
\[
W_p(f,g) = \left(\inf_{\pi \in \Pi(f,g)} \int_{\R^2\times\R^2} |x-y|^p\, \dd\pi(x,y)\right)^{\frac1p}.
\]
In this paper, we will consider the cases $p=1$ and $p=2$ only. Both are ordered in the sense that
\begin{equation}\label{12}
W_1(f,g) \le \|f\|_{L^1}^{\frac12} W_2(f,g)
\end{equation}
by Jensen's inequality, where we used the fact that $\pi[\R^2\times \R^2] = \|f\|_{L^1}$.

As a consequence of the Kantorovich--Rubinstein duality theorem  
\begin{equation}\label{9}
W_1(f,g) = \sup\left\{ \int_{\R^2} (f-g)\zeta\dx:\: \|\grad \zeta\|_{L^{\infty}}\le 1\right\},
\end{equation}
cf.~\cite[Theorem 1.14]{Villani03}, the $1$-Wasserstein distance is a transshipment cost that only sees the difference of the marginals, and thus, $W_1$ can be naturally extended as a measure on the space of not necessarily nonnegative configurations with same spatial average \eqref{10}. This is particularly convenient in our application to solutions to the Navier--Stokes equations, as these conserve the spatial average but not the $L^1$ norm. 

Concerning $W_2$, we will use the fact that the Wasserstein distance dominates the $\dot H^{-1}$ norm,
\begin{equation}\label{13}
\|f-g\|_{\dot H^{-1}} \le \max\{\|f\|_{L^{\infty}},\|g\|_{L^{\infty}}\}^{\frac12} W_2(f,g),
\end{equation}
cf.~\cite[Theorem 2.9]{Loeper06a}. 


%
%
%

\section{Proofs}
Due to the fact that the vorticity fields are not necessarily nonnegative and our argument is essentially based on $W_2$-distances, which are defined for nonnegative quantities only, we need the following construction.

For a general initial vorticity distribution $\omega_0$, we  consider separately the evolution of the positive and negative parts given by the linear equations
\begin{gather}
\partial_t \on_{\pm} + \un \cdot\grad \on_{\pm}   = \nu \laplace \on_{\pm},\quad \on_{\pm} (0)= \omega_0^{\pm},\label{101}\\
\partial_t \omega_{\pm} + u\cdot\grad\omega_{\pm} =0,\quad \omega_{\pm} (0)= \omega_0^{\pm},\label{102}
\end{gather}
where the superscript plus and minus signs indicate the positive and negative parts, i.e., $\omega_0^+ = \max\{0,\omega_0\}$ and $\omega_0^- = \max\{0,-\omega_0\}$, while the subscript plus and minus signs just mark the solutions. Notice that both equations are well-posed (and so is the hypoelliptic equation in \eqref{3}, which is a hybrid version of \eqref{101} and \eqref{102}) as a consequence of the DiPerna--Lions theory \cite{DiPernaLions89}, because $\grad u$ and $\grad \un$ belong both to the class $L^{\infty}(\R_+;L^p(\R^2)))$ for any $p\in [1,\infty)$ by the assumptions on the initial vorticity \eqref{6}, the a priori estimates in \eqref{104} and \eqref{105}, and Calder\'on--Zygmund theory.

By the maximum principles for the respective equations, the solutions are nonnegative. Moreover, as both equations are conservative thanks to the incompressibility condition $\div\un=\div u=0$, the total masses are preserved, $\|\on_\pm(t)\|_{L^1} = \|\omega_{\pm}(t)\|_{L^1} = \|\omega_0^{\pm}\|_{L^1}$. Finally, by uniqueness and linearity, it holds that $\on = \on_{+}-\on_-$ and $\omega = \omega_+-\omega_-$. We then have by the triangle inequality, \eqref{9} and \eqref{12} that
\begin{align*}
W_1(\on,\omega) &\le W_1(\on_+,\omega_+) + W_1(\on_-,\omega_-) \\
&\le \|\omega_0^+\|_{L^1}^{\frac12} W_2(\on_+,\omega_+) + \|\omega_0^-\|_{L^1}^{\frac12}W_2(\on_-,\omega_-)\\
&\le \|\omega_0\|_{L^1}^{\frac12}\left(W_2(\on_+,\omega_+)  +W_2(\on_-,\omega_-)\right).
\end{align*}
An analogous estimate holds true for the $\dot H^{-1}$ norm via \eqref{13}, namely
\[
\|\omega^\nu-\omega\|_{\dot H^{-1}} \le \|\omega_0\|_{L^{\infty}}^{\frac12}   \left(W_2(\on_+,\omega_+)  +W_2(\on_-,\omega_-)\right).
\]

Theorems \ref{T2} and \ref{T1} are thus  consequences of the following result.

\begin{theorem}\label{T3}
For $\nu\ll1 $ and $t\ll \frac1{|\log \nu|}$, it holds that
\[
\left(W_2(\on_+(t),\omega_+(t))  +W_2(\on_-(t),\omega_-(t))\right)\lesssim  \sqrt{\nu t} .
\]
Moreover, for  any  fixed $t>0$, it holds
\[
\left(W_2(\on_+(t),\omega_+(t))  +W_2(\on_-(t),\omega_-(t))\right) = O\left( \left(\frac{\nu}{|\log\nu|}\right)^{\frac12\exp(-Ct)}\right)
\]
as $\nu\to0$, where $C >0$ is a constant dependent only on $\|\omega_0\|_{L^1}$ and $\|\omega_0\|_{L^{\infty}}$. 
\end{theorem}

We now turn to the proof of Theorem \ref{T3}, in which we roughly follow and extend  Loeper's stability estimate for the Euler equations \cite{Loeper06a}. Loeper's proof is  based on the Lagrangian formulation of the advection equation \eqref{2}. Its viscous version leads to the stochastic differential equation
\[
\dd X_t(x) = \un(t,X_t(x))\dt + \sqrt{2\nu} \dd W_t, \quad X_t(x)=x,
\]
and the Lagrangian representation of the vorticity, $\on(t) = \E[\omega_0\circ X_t^{-1}]$. Instead of working with the stochastic flow, we propose a deterministic (or Eulerian) derivation of the stability-type estimate   in Theorem \ref{T3} via the coupling method, see, e.g., \cite{FournierPerthame19}. For that purpose, we choose  functions $\eta_0^{\pm} \in \Pi(\omega_0^{\pm},\omega_0^{\pm})$ and consider the linear hypoelliptic advection-diffusion equations
\begin{equation}\label{3}
\partial_t \eta_{\pm} + u(x)\cdot \grad_x\eta_{\pm} + \un(y)\cdot \grad_y \eta_{\pm} = \nu \laplace_y \eta_{\pm}
\end{equation}
 on the product space $\R^2\times\R^2$, with initial conditions $\eta_{\pm}(0) = \eta_0^{\pm}$. By construction, the marginals of $\eta_{\pm}$ coincide with the unique solutions of equations \eqref{101} and \eqref{102},
  \[
 \int_{\R^2}\eta_{\pm}(t,x,y)\, \dd x = \on_{\pm}(t,y),\quad \int_{\R^2} \eta_{\pm}(t,x,y)\, \dd y = \omega_{\pm}(t,x),
 \]
and thus, by definition of the Wasserstein distance, 
\begin{equation}\label{14}
W_2^2(\on_{\pm}(t),\omega_{\pm}(t)) \le \iint_{\R^2\times \R^2} |x-y|^2\eta_{\pm}(t,x,y)\dx\dy =:Q_{\pm}(t).
\end{equation}
We also set
\[
Q(t) = Q_+(t) + Q_-(t).
\]

By standard approximation procedures, we may assume that $\eta_0^{\pm}$ is smooth and compactly supported so that $Q_{\pm}(0)$ is finite, and that $Q_{\pm}(t)$ is smooth. In fact, since the Wasserstein distance between the Navier--Stokes and Euler vorticities is initially vanishing by \eqref{6}, we may choose $\eta_0$ such that $Q(0)$ is arbitrarily small, say
\begin{equation}\label{20}
Q(0)\left( 1+\log \left(1+\frac1{Q(0)}\right)\right) < \nu.
\end{equation}

Our first goal is the following differential inequality.

\begin{lemma}\label{L1}
It holds that
\begin{equation}\label{5}
\frac{\dd Q}{\dd t} \lesssim Q\left( 1+\log \left(1+\frac1Q\right)\right) + \nu.
\end{equation}
\end{lemma}

\begin{proof}In the following computation, we neglect the time dependences of the involved functions.  We first derive simultaneously estimates on the rates of change of $Q_+$ and $Q_-$. Differentiation,  (multiple) integration by parts, and the fact that the velocity fields are both divergence-free yield 
\begin{align*}
\frac{\dd Q_{\pm}}{\dd t} & = 2 \iint_{\R^2\times\R^2} (x-y)\cdot\left(u(x)-\un(y)\right)\eta_{\pm}(x,y)\dx\dy + 2\nu\iint_{\R^2\times\R^2} \eta_{\pm}(x,y)\dx\dy.
\end{align*}
Because \eqref{3} can be put  in conservation form (because $u$ and $\un$ are divergence-free), and in view of the marginal conditions for $\eta_0^{\pm}$, we notice that
\[
\iint_{\R^2\times\R^2} \eta_{\pm}(x,y)\dx\dy = \iint_{\R^2\times\R^2} \eta_0^{\pm}(x,y)\dx\dy = \|\omega^{\pm}_0\|_{L^1}.
\]
We thus have and write
\begin{align*}
\frac{\dd Q_{\pm}}{\dd t} & = 2\iint_{\R^2\times\R^2 }(x-y)\cdot \left(u(x) - u(y)\right)\eta_{\pm}(x,y)\dx\dy \\
&\qquad+2\iint_{\R^2\times\R^2 }(x-y)\cdot \left(u(y) - \un(y)\right)\eta_{\pm}(x,y)\dx\dy + 2\nu \|\omega_0^{\pm}\|_{L^1}\\
&=: I_{\pm}^1 + I_{\pm}^2 + 2\nu \|\omega^{\pm}_0\|_{L^1}.
\end{align*}

In order to estimate the first integral term, we use the $\log$-Lipschitz estimate \eqref{4} for the velocity field,
\[
|I_{\pm}^1| \lesssim Q_{\pm} + \iint_{\R^2\times\R^2} |x-y|^2\log\left(1+\frac1{|x-y|}\right)\eta_{\pm}(x,y)\dx\dy.
\]
Because $s\mapsto s\log\left(1+\frac1s\right)$ is concave, we furthermore have with Jensen's inequality
\begin{align*}
\MoveEqLeft{\iint_{\R^2\times\R^2} |x-y|^2\log\left(1+\frac1{|x-y|}\right)\eta_{\pm}(x,y)\dx\dy}\\
 &\lesssim \left( \iint_{\R^2\times\R^2} |x-y|^2 \eta_{\pm}(x,y)\dx\dy\right) \log\left(1+ \frac{\iint_{\R^2\times\R^2} |x-y|\eta_{\pm}(x,y)\dx\dy}{\iint_{\R^2\times\R^2} |x-y|^2\eta_{\pm}(x,y) \dx\dy}\right)\\
&\le Q_{\pm} \log\left(1+ \left(\frac{\|\omega_0^{\pm}\|_{L^1}}{Q_{\pm}}\right)^{\frac12}\right)\\
&\lesssim Q_{\pm}\left(1+ \log\left(1+\frac1Q_{\pm}\right)\right).
\end{align*}
It thus follows that
\[
I_{\pm}^1\lesssim Q_{\pm}\left(1+ \log\left(1+\frac1Q_{\pm}\right)\right).
\]

For the second integral term, we use the Cauchy--Schwarz inequality and the marginal condition for $\eta_{\pm}$,
\begin{align*}
|I_{\pm}^2|  & \le \left(\iint_{\R^2\times\R^2} |x-y|^2 \eta_{\pm}(x,y)\dx\dy\right)^{\frac12} \left(\iint_{\R^2\times\R^2} |u(y)-\un(y)|^2 \eta_{\pm}(x,y)\dx\dy\right)^{\frac12}\\
& = Q_{\pm}^{\frac12} \left( \int_{\R^2} |u(y)-\un(y)|^2 \on_{\pm}(y)\dy\right)^{\frac12}\\
&\le Q_{\pm}^{\frac12} \|\on_{\pm}\|_{L^{\infty}}^{\frac12}\|\un-u\|_{L^2}.
\end{align*}
It remains to notice that $\|\on_{\pm}\|_{L^{\infty}}\le \|\omega_0\|_{L^{\infty}}$. Moreover, the $L^2$ norm of the velocity difference is the $\dot H^{-1}$ norm of the vorticity difference, which is bounded by the $2$-Wasserstein distance, cf.~\eqref{13}. Hence, via the triangle inequality,
\[
|I_{\pm}^2| \lesssim Q_{\pm}^{\frac12} \left(W_2(\omega_+,\on_+) +W_2(\omega_-,\on_-)\right)\le Q_{\pm}^{\frac12}Q^{\frac12}.
\]

Combining the previous estimates yields
\[
\frac{\dd Q}{\dt} \lesssim Q_{+}^{\frac12}Q^{\frac12}+Q_{-}^{\frac12}Q^{\frac12} + Q_{+}\left(1+ \log\left(1+\frac1Q_{+}\right)\right)+Q_{-}\left(1+ \log\left(1+\frac1Q_{-}\right)\right) +\nu,
\]
which implies the statement of the lemma, because $s\mapsto s\left(1+\log\left(1+\frac1s\right)\right)$ is an increasing function on $\R_+$.
\end{proof}

It remains to integrate the differential inequality \eqref{5}.

\begin{lemma}\label{L2}
For $\nu\ll1 $ and $t\ll \frac1{|\log \nu|}$, it holds that
\[
Q(t) \lesssim  Q(0)+ \nu t .
\]
Moreover, for  any  fixed $t>0$, it holds
\[
Q(t)  = O\left( \left(Q(0)+ \frac{\nu}{|\log\nu|}\right)^{\exp(-Ct)}\right)
\]
as $\nu\to0$, where $C >0$ is a constant dependent only on $\|\omega_0\|_{L^1}$ and $\|\omega_0\|_{L^{\infty}}$. 
\end{lemma}

\begin{proof}The argument is very elementary. We provide it for the convenience of the reader.

We remark that by replacing $Q(t)$ by $Q(t)+\nu t$, which satisfies the same differential inequality \eqref{5} because $s\mapsto s\left(1+\log\left(1+\frac1s\right)\right)$ is an increasing function, we may without loss of generality assume that
\begin{equation}\label{100}
Q(t)\ge \nu t.
\end{equation}

As a consequence of  \eqref{20}, there exists a time $t_1$ up to which  $Q(t)$ is  small in the sense that 
\[
Q\left(1+\log\left(1+ \frac1Q\right)\right) \le \nu ,
\]
and thus, \eqref{5} reduces to
\[
\frac{\dd Q}{\dd t} \lesssim \nu,
\]
which yields that
\[
Q(t)\lesssim Q(0) +  \nu t\quad \mbox{for }t\in[0,t_1].
\]
This already proves the first statement. We only have to give an estimate on $t_1$. For this, we notice that we  may furthermore assume that $Q(0) \le \nu t_1$, so that $Q(t_1)\sim \nu t_1$ thanks to \eqref{100}. Then, the final time $t_1$ is given  by the estimate
\begin{equation}\label{103}
\nu t_1 \left(1+ \log\left(1+ \frac1{\nu t_1}\right)\right) \sim \nu.
\end{equation}
For $\nu\ll1 $, this relation and,  more precisely, the behavior of the function on the left-hand side enforce $\nu t_1\ll1$. In particular, \eqref{103} can be restated as
\[
\frac1{t_1} \sim \log \frac1{\nu t_1},
\]
which entails that $t_1\ll1$. Therefore, $1/t_1$ is much larger than its logarithm, and, thus, by the product rule for the logarithm, we must have
\[
\frac1{t_1}\sim \log\frac1{\nu}.
\]

We now turn to the second estimate of Lemma \ref{L2}. Since we are interested in an asymptotic statement for fixed times, we may always assume that $Q<1$. Moreover, it is enough to consider the case
\[
Q\left(1+\log\left(1+ \frac1Q\right)\right) \ge \nu 
\]
for $t\ge t_1$, because otherwise $Q$ would have to decrease in time. In this situation, \eqref{5} simplifies to
\[
\frac{\dd Q}{\dd t} \lesssim Q\left(1+\log \frac1Q\right) .
\]
This differential inequality can be rewritten as
$
C\log Q+ \frac{\dd }{\dd t}\log Q \le C,
$
for some $C>0$, and thus
$
\frac{\dd}{\dd t} \left(e^{C t} \log Q\right) \le Ce^{Ct}.
$
 A short computation reveals that
\[
Q(t) \lesssim Q(t_1)^{\exp(C(t_1-t))}\lesssim \left(\frac{\nu}{\log\frac1{\nu}}+Q(0)\right)^{\exp(-Ct)},
\]
where we have used the above estimate on $Q(t_1)$ and the estimate for $t_1$.
%
\end{proof}

\begin{proof}[Proof of Theorem \ref{T3}]
From   \eqref{14} and Lemma \ref{L2}, we deduce that
\[
W_2^2(\on_+(t),\omega_+(t))+W_2^2(\on_-(t),\omega_-(t)) \lesssim Q(0)+ \nu t
\]
for times $t\ll \frac1{|\log\nu|}$. It remains to notice that we can set $Q(0)=0$ by optimizing over $\eta_0^{\pm}\in \Pi(\omega_0^{\pm},\omega_0^{\pm})$. This concludes the proof of the first statement. The second one follows analogously.
\end{proof}

\section*{Acknowledgement} 
This work is funded by the Deutsche Forschungsgemeinschaft (DFG, German Research Foundation) under Germany's Excellence Strategy EXC 2044 --390685587, Mathematics M\"unster: Dynamics--Geometry--Structure. The author thanks Tarek Elgindi for pointing out an error in the first draft of this article.

\bibliography{euler}

\begin{thebibliography}{10}

\bibitem{AbidiDanchin04}
{\sc Abidi, H., and Danchin, R.}
\newblock Optimal bounds for the inviscid limit of {N}avier-{S}tokes equations.
\newblock {\em Asymptot. Anal. 38}, 1 (2004), 35--46.

\bibitem{Ben-Artzi94}
{\sc Ben-Artzi, M.}
\newblock Global solutions of two-dimensional {N}avier-{S}tokes and {E}uler
  equations.
\newblock {\em Arch. Rational Mech. Anal. 128}, 4 (1994), 329--358.

\bibitem{Chemin96}
{\sc Chemin, J.-Y.}
\newblock A remark on the inviscid limit for two-dimensional incompressible
  fluids.
\newblock {\em Comm. Partial Differential Equations 21}, 11-12 (1996),
  1771--1779.

\bibitem{ConstantinDrivasElgindi19}
{\sc Constantin, P., Drivas, T.~D., and Elgindi, T.~M.}
\newblock Inviscid limit of vorticity distributions in yudovich class.
\newblock Preprint arXiv:1909.04651, 2019.

\bibitem{ConstantinWu95}
{\sc Constantin, P., and Wu, J.}
\newblock Inviscid limit for vortex patches.
\newblock {\em Nonlinearity 8}, 5 (1995), 735--742.

\bibitem{ConstantinWu96}
{\sc Constantin, P., and Wu, J.}
\newblock The inviscid limit for non-smooth vorticity.
\newblock {\em Indiana Univ. Math. J. 45}, 1 (1996), 67--81.

\bibitem{Cozzi09}
{\sc Cozzi, E.}
\newblock Vanishing viscosity in the plane for nondecaying velocity and
  vorticity.
\newblock {\em SIAM J. Math. Anal. 41}, 2 (2009), 495--510.

\bibitem{Cozzi14}
{\sc Cozzi, E.}
\newblock Vanishing viscosity in the plane for nondecaying velocity and
  vorticity, {II}.
\newblock {\em Pacific J. Math. 270}, 2 (2014), 335--350.

\bibitem{DiPernaLions89}
{\sc DiPerna, R.~J., and Lions, P.-L.}
\newblock Ordinary differential equations, transport theory and {S}obolev
  spaces.
\newblock {\em Invent. Math. 98}, 3 (1989), 511--547.

\bibitem{FournierPerthame19}
{\sc Fournier, N., and Perthame, B.}
\newblock Monge-{K}antorovich distance for {PDEs}: the coupling method.
\newblock Preprint arXiv:1903.11349, 2019.

\bibitem{Yudovich63}
{\sc Judovi\v{c}, V.~I.}
\newblock Non-stationary flows of an ideal incompressible fluid.
\newblock {\em \v{Z}. Vy\v{c}isl. Mat i Mat. Fiz. 3\/} (1963), 1032--1066.

\bibitem{Loeper06a}
{\sc Loeper, G.}
\newblock {Uniqueness of the solution to the {V}lasov-{P}oisson system with
  bounded density}.
\newblock {\em J. Math. Pures Appl. (9) 86}, 1 (2006), 68--79.

\bibitem{MajdaBertozzi02}
{\sc Majda, A.~J., and Bertozzi, A.~L.}
\newblock {\em {Vorticity and incompressible flow}}, vol.~27 of {\em {Cambridge
  Texts in Applied Mathematics}}.
\newblock Cambridge University Press, Cambridge, 2002.

\bibitem{Masmoudi07}
{\sc Masmoudi, N.}
\newblock Remarks about the inviscid limit of the {N}avier-{S}tokes system.
\newblock {\em Comm. Math. Phys. 270}, 3 (2007), 777--788.

\bibitem{Seis17}
{\sc Seis, C.}
\newblock A quantitative theory for the continuity equation.
\newblock {\em Ann. Inst. H. Poincar\'{e} Anal. Non Lin\'{e}aire 34}, 7 (2017),
  1837--1850.

\bibitem{Seis18}
{\sc Seis, C.}
\newblock Optimal stability estimates for continuity equations.
\newblock {\em Proc. Roy. Soc. Edinburgh Sect. A 148}, 6 (2018), 1279--1296.

\bibitem{Villani03}
{\sc Villani, C.}
\newblock {\em Topics in optimal transportation}, vol.~58 of {\em Graduate
  Studies in Mathematics}.
\newblock American Mathematical Society, Providence, RI, 2003.

\end{thebibliography}
\bibliographystyle{acm}

\flushleft{\vspace{5em}
\small Christian Seis\\
Institut f\"ur Analysis und Numerik,  Westf\"alische Wilhelms-Universit\"at M\"unster,\\
Orl\'eans-Ring 10, 48149 M\"unster, Germany.\\
\emph{E-mail address}: \texttt{seis@wwu.de}
}
\end{document}